  \def\misajour{December 21,  2010} %

\documentclass[10pt]{article}  

\usepackage{
amssymb, 
amsthm,
amsmath,
}

\usepackage[applemac]{inputenc}
\usepackage[pdftex]{graphicx}
\usepackage{hyperref} 
\usepackage{xcolor}

\def\atop#1#2{
\genfrac{}{}{0pt} {} 
{#1} 
{#2}}

\newtheorem{theorem}[equation]{Theorem}
\newtheorem{proposition}[equation]{Proposition}
\newtheorem{lemma}[equation]{Lemma}
\newtheorem{corollary}[equation]{Corollary}
\newtheorem{conjecture}[equation]{Conjecture}

\def\rmh{{\mathrm{h}}}
\def\rmL{{\mathrm{L}}}
\def\rk{{\mathrm{rk}}}

\def\calS{{\mathcal{S}}}
\def\calU{{\mathcal{U}}}
\def\calV{{\mathcal{V}}}
\def\uc{\underline{c}}
\def\un{\underline{n}}
\def\us{\underline{s}}
\def\ut{\underline{t}}
\def\uz{\underline{z}}
\def\uX{\underline{X}}
\def\ulambda{\underline{\lambda}}
\def\ugamma{\underline{\gamma}}

\def\bC{{\mathbb{C}}}
\def\bG{{\mathbb{G}}}
\def\bP{{\mathbb{P}}}
\def\bQ{{\mathbb{Q}}}
\def\bR{{\mathbb{R}}}
\def\bZ{{\mathbb{Z}}}

\def\Gammabar{\overline{\Gamma}}
\def\Qbar{\overline{\bQ}}
\def\Lie{{\mathrm{Lie}}}

\begin{document}
 
   \hfill
    
    \null
    \vskip -3 true cm
    
    \noindent
{
 }
 
 \noindent
 Afrika Matematika, to appear
  \hfill
  {\it 
   \misajour}

     \vskip 3 true cm

    \begin{center}
 {\bf  
On the $p$-adic closure of a subgroup of rational points 
\\
on an Abelian variety}

 \smallskip
by

 \smallskip
{ \it
 Michel Waldschmidt\footnote{Université Pierre et Marie Curie (Paris 6), Paris, France}
 }

 \end{center}

  \medskip

\section*{Abstract}
In 2007, B.~Poonen (unpublished) studied the $p$--adic closure of a subgroup of rational points  on a commutative algebraic group. More recently, J.~Bellaïche asked the same question for the special case of Abelian varieties. These problems are $p$--adic analogues of a question raised earlier by B.~Mazur on the density of rational points for the real topology. For a simple Abelian variety over the field of rational numbers, we show that the actual $p$--adic rank is at least the third of the expected value. 

\bigskip
\noindent{\bf Acknowledgments} The  author wishes to take this opportunity to thank Jean Fresnel, who introduced him to $p$-adic transcendence problems long back. This research started thanks to a discussion with Bjorn Poonen in Tucson during the Arizona Winter School in March 2008. Further discussions on this subject with Cristiana Bertolin in Regensburg shortly afterwards were also useful. The motivation to write this paper was renewed by a correspondence with Joël Bellaïche early 2010 \cite{Bellaiche}, while the  author was visiting the Harish-Chandra Research Institute in Allahabad, where he had fruitful discussions with Chandan Singh Dalawat.

\section{Introduction} \label{Section:Introduction}
 
Let  $A$ be a simple Abelian variety over $\bQ$ of dimension $g$, $\Gamma$ a subgroup of $A(\bQ)$ of rank $\ell$ over $\bZ$, $p$ a prime number, 
$\log:A(\bQ_p)\rightarrow T_A(\bQ_p)$ the canonical map from the $p$--adic Lie group $A(\bQ_p)$ to the  $p$--adic Lie algebra $T_A(\bQ_p)$ (see \S~\ref{Subsection:p-adicLogarithm}) and $r$ the dimension of the $\bZ_p$--space spanned by $\log \Gamma$ in $T_A(\bQ_p)$. We have $r\le \min\{g, \ell\}$. 

\begin{conjecture}\label{Conjecture}
Under these hypotheses, $r=\min\{g, \ell\}$. 
\end{conjecture}

This conjecture trivially holds for an elliptic curve ($g=1$). 

 The real analog of this conjecture is related with a conjecture of B.~Mazur
\cite{MR1330932}. See also the conjectures by Yves André
\cite{MR990016,MR2115000}.

\begin{theorem} \label{Theorem:MainTranscendenceResult}
We have 
$$
r\ge \frac{\ell g}{\ell+2g}\cdotp
$$

\end{theorem}

\begin{corollary}
Under the same assumptions, 
$$
r\ge\frac{1}{3} \min\{g,\ell\}.
$$
Moreover, if $\ell>2g(g-1)$, then $r=g$.

\end{corollary}

Theorem \ref{Theorem:MainTranscendenceResult} is a special case of Theorem 2.1 of \cite{MR84k:10032}, where the simple Abelian variety $A$ over $\bQ$ is replaced by a commutative algebraic group $G$ over a number field. 
Our special case enables us to produce a much simpler proof. In particular, the zero estimate is much easier here, since there is no algebraic subgroup of $G$ to be taken care of. Also, the main difference between our proof and the two proofs in  \cite{MR84k:10032} is that we use an interpolation determinant in place of an auxiliary function (Proposition 2.7 of  \cite{MR84k:10032}) or in place of an auxiliary functional (Proposition 2.10 of  \cite{MR84k:10032}): we do not need the $p$--adic Siegel Lemma (Lemma 3.3 of  \cite{MR82k:10042}). The two proofs in  \cite{MR84k:10032} are dual to each other, and this duality is just a transposition of the interpolation determinant of the present paper.

\section{Further notations and auxiliary results}\label{Section:FurtherNotations}

We keep the notations of \S~\ref{Section:Introduction}. We select $\ell$ elements $\gamma_1,\ldots,\gamma_\ell$ in $\Gamma$ linearly independent over $\bZ$. 

For $T$ a positive integer, we denote by $\bZ^g(T)$ the set of tuples $\ut=(t_1,\ldots,t_g)$ in $\bZ^g$ with $0\le t_i<T$ ($1\le i\le g)$. Similary, for $S\in\bZ_{>0}$, 
$\bZ^\ell(S)$ denotes the set of tuples $\us=(s_1,\ldots,s_\ell)$ in $\bZ^\ell$ with $0\le s_j<S$ ($1\le j\le \ell)$. Further, $\Gamma(S)$ will denote the set of $s_1\gamma_1+\cdots+s_\ell\gamma_\ell$ with $\us\in\bZ^\ell(S)$. Hence $\Gamma(S)$ is a subset of $A(K)$ with $S^\ell$ elements.

\subsection{The $p$--adic logarithm}\label{Subsection:p-adicLogarithm}

We follow the paper by B.~Poonen \cite{Poonen} which refers to N.~Bourbaki \cite{MR1728312} Chap.~III, \S~1 and \S~7.6.

Since $\Gamma$ is a finitely abelian subgroup of $A(\bQ_p)$ of rank $\ell$, $\log\Gamma$ is also a finitely generated abelian subgroup of $T_A(\bQ_p)$ of the same rank $\ell$ over $\bZ$. The closure $\overline{\log\Gamma}=\log\Gammabar$ with respect to the $p$-adic topology is nothing else than  the $\bZ_p$--submodule of $T_A(\bQ_p)$ spanned by $\log\Gamma$, hence is a finitely generated $\bZ_p$--module. The dimension of $\Gammabar$ as a Lie group over $\bQ_p$ is 
$$
\dim \Gammabar:= \rk_{\bZ_p}\overline{\log\Gamma}.
$$

\subsection{Heights}

\subsubsection{A projective embedding}\label{Subsubsection:ProjectiveEmbedding}

We fix an embedding $\iota$ of the Abelian variety $A$ into a projective space $\bP_N$ over $\bQ$, with an image which is not contained into the hyperplane $X_0=0$, and so that the functions $X_1/X_0,\ldots,X_g/X_0$ are algebraically independent over $A$ (recall that $A$ has dimension $g$). We also assume that for $\us\in\bZ^\ell$,  $\iota(\gamma_{\us})$ does not lie in the hyperplane $X_0=0$ and we  denote by $(1:\gamma_{\us 1}:\cdots:\gamma_{\us N})$ the coordinates of $\iota(\gamma_{\us})$ in $\bP_N$, so that $\gamma_{ \us \nu}\in \bQ$ for $1\le \nu\le N$ and $\us\in\bZ^\ell$. For convenience, we also assume that the zero element of $A$ has projective coordinates $(1:0:\cdots:0)$.

\subsubsection{Absolute logarithmic height}\label{Subsubsection:Height}


 Denote by $P=\{2,3,5,\dots\}$ the set of positive prime numbers and by $M_\bQ $  the set of normalized places of $\bQ$ indexed by $P\cup\{\infty\}$:
 for $c\in\bQ^\times$ we write 
 $$
 c=\pm \prod_{p\in P} p^{v_p(c)}
 $$
 and we have 
 $$
 \begin{cases}
|c|_v=|c|=\max\{c,-c\}& \hbox{ for $v=v_\infty$}\\
 |c|_p = p^{-v_p(c)}& \hbox{ for $ p\in P$}.
 \end{cases}
 $$
  The {\it product formula}, in this very simple case,  
 states that, for $c\in \bQ\setminus\{0\}$, 
$$
 \prod_{v\in M_\bQ}    |c|_v=1.
 $$
 The absolute logarithmic height of $c\in \bQ$ is defined as
 $$
\rmh(c)= \sum_{v\in M_\bQ}    \log\max\{1,\; |c|_v\}.
$$  
For $c\in \bQ^\times$, we write $c=a/b$ where $a\in\bZ\setminus\{0\}$ and $b\in\bZ_{>0}$ are two relatively prime integers. Since 
$\min\{v_p(a),\; v_p(b)\}=0$ for all $p\in P$, we have, for all $p\in P$,  
$$
\max\{|a|_p,|b|_p\}=1, 
\quad\hbox{which means}\quad 
\max\{1,|c|_p\}= |b|_p^{-1}.
$$
Hence, by the product formula,
$$
 \prod_{p\in P}   \max\{1,|c|_p\}= b.
 $$
 Multiplying both sides by $\max\{1,|c|\}$ yields 
 $$
\rmh(c)=\log\max\{|a|,b\},
$$
which can be taken as an alternative definition for the absolute logarithmic height. 

 Liouville's inequality is very simple in this context: 

\begin{lemma}\label{Lemma:InegaliteLiouville}
If $c$ is a non--zero rational number and $p$ a prime number, then 
$$
  \log|c|_p\ge - \rmh (c).
$$
\end{lemma}

For  $N\ge 1$ and $\uc=(c_0:\dots:c_N)\in\bP_N(\bQ)$, we set
$$
\rmh(\uc)=  \sum_{v\in M_\bQ}   \log\max\{ |c_0|_v,\;\dots,\;  |c_N|_v\}.
$$
If $c_0,\ldots,c_N$ are rational integers, not all of which are zero, which are relatively prime, then 
$$
\rmh(\uc)=  \log\max\{ |c_0|,\;\dots,\;  |c_N|\}.
$$
Notice that for $c\in\bQ$, $\rmh(c)=\rmh(1:c)$. 
 
\subsubsection{Néron--Tate height  }\label{Subsubsection:NeronTate}

The projective embedding considered in \S~\ref{Subsubsection:ProjectiveEmbedding} is associated with a very ample line bundle 
on $A$, to which is associated a canonical height which is a quadratic function  (see  \cite{MR1757192} Chap.~3 and  \cite{MR1745599} \S~B.5). 

\begin{lemma}\label{Lemma:NeronTateHeight}
For $\us\in\bZ^{\ell}(S)$,
$$
\rmh(s_1\gamma_1+\cdots+s_{\ell} \gamma_\ell)=\rmh(1:\gamma_{\us 1}:\cdots:\gamma_{\us N})\le c S^2.
$$
\end{lemma}

\subsubsection{Upper bound for the height  }\label{Subsubsection:UpperBoundHeight}

We shall use the following result, which is a very simple case of  Lemma {3.8.} in \cite{MR1756786} (where $\bQ$ is replaced by a number field).
We denote by  $\rmL(f)$  the length of a polynomial  $f$ (sum of the absolute values of the coefficients).

\begin{lemma} \label{Lemma:Height}
Let   $\nu_1,\ldots,\nu_L$ be positive integers. For  $1\le i\le L$, let $\gamma_{i1},\ldots,\gamma_{i\nu_i}$ be rational numbers. Denote by $\ugamma$ the point $\bigl(\gamma_{ij}\bigr)_{1\le j\le \nu_i,1\le i\le L}$ in $\bQ^{\nu_1+\cdots+\nu_L}$. Further, let $f$ be a nonzero polynomial in $\nu_1+\cdots+\nu_L$ variables, with coefficients in $\bZ$, of total degree at most $N_i$ with respect to the $\nu_i$ variables corresponding to $\gamma_{i1},\ldots,\gamma_{i\nu_i}$.  Then 
$$
\rmh\bigl(f(\ugamma)\bigr)\le\log \rmL(f)+
\sum_{i=1}^L N_i\rmh(1\colon \gamma_{i1}\colon \cdots\colon \gamma_{i\nu_i}). 
$$
\end{lemma}
 
 \begin{proof}
 Let us write 
 $$
 f(\uX)=\sum_{\ulambda} c_{\ulambda}  \prod_{i=1}^L \prod_{j=1}^{\nu_i} X_{ij}^{\lambda_{ij}},
 $$
 where $\uX$ (resp. $\ulambda$) stands for  the $\nu_1+\cdots+\nu_L$ --tuple  $(X_{ij})_{1\le j\le \nu_i,\; 1\le i\le L}$ (resp.  $(\lambda_{ij})_{1\le j\le \nu_i,\; 1\le i\le L}$. Lemma \ref{Lemma:Height} follows from the estimates
 \begin{align}\notag 
 |f(\ugamma)|
 &\le\sum_{\ulambda}   |c_{\ulambda}| \prod_{i=1}^L \prod_{j=1}^{\nu_i}\max\{1,\;  |\gamma_{ij}|\}^{\lambda_{ij}}\\
 \notag
 &\le \rmL(f)  \prod_{i=1}^L  \max\{1,\;    |\gamma_{i1}| ,\;  \dots , \;  |\gamma_{i\nu_i}|\}^{N_i}
\end{align}
 and
\begin{align}\notag 
 |f(\ugamma)|_p
 &\le\max_{\ulambda}   \prod_{i=1}^L \prod_{j=1}^{\nu_i}\max\{1,\;  |\gamma_{ij}|_p\}^{\lambda_{ij}}\\
 \notag
 &\le  \prod_{i=1}^L  \max\{1,\;   |\gamma_{i1}|_p,\;  \dots ,\;  |\gamma_{i\nu_i}|_p\}^{N_i}
\end{align}
 for $p\in P$. 
 \end{proof}

\subsection{$p$--adic analytic functions}\label{Subsection:p-adicFunctions}

\subsubsection{Ultrametric power series}\label{Subsubsection:UltrametricPowerSeries}

We follow 
\cite{MR1942136}. The field $\bQ_p$ 
 is complete for the $p$--adic absolute value. Let 
$$
f=\sum_{n_1\ge 0}\cdots\sum_{n_r\ge 0}  a_{n_1,\ldots,n_r} z_1^{n_1}\cdots z_r ^{n_r}=\sum_{\un\in\bZ_{\ge 0}^r}  a_{\un} \uz^{\un}
$$
be a formal series with coefficients in $\bQ_p$. If $R$ is a real number $>0$, we set
$$
 |f|_R=\sup_{\un\in\bZ_{\ge 0}^r}   R^{|\un|} |a_{\un}|, \quad\hbox{where } \quad |\un|= n_1+\cdots+n_r.
 $$
 We have
 $$
 |f+g|_R\le\sup\{|f|_R,\; |g|_R\},
 \quad |\lambda f |_R=|\lambda| \cdot |f|_R
\quad\hbox{ 
 and }\quad
 |fg|_R=|f|_R |g|_R  
 $$
 if $|f|_R$ and $|g|_R$ are finite.  
 When $|f|_R$ is finite, the series $f(\uz)$ converges in the polydisc $|z_i|<R$. Moreover, it converges in the closed polydisc  $|z_i|\le R$ when $R^{|\un|} |a_{\un}|$ tends to zero. We have
 $$
 |f(\uz)|\le |f|_R.
 $$
 Since the residue field of $\bQ_p$ is infinite and the group of values of $\bQ_p^\times$ is dense, we also have
 $$
 |f|_R=\sup |f(\uz)| \quad  \hbox{ for } \quad |z_i|<R.
 $$
If $R'\le R$, we have $|f|_{R'}\le |f|_R$ ({\it maximum modulus principle}).

\subsubsection{Ultrametric Schwarz Lemma}\label{Subsubsection:UltrametricSchwarzLemma}

 The purpose of the Schwarz's Lemma is to improve the maximum modulus principle by taking into account the zeros of $f$ inside the polydisc  $ |z_i|<R'$. With the method of interpolation determinants of Laurent \cite{MR1756786}, we need only to take into account the multiplicity of the zero at the origin. For this reason, the proof reduces to the one variable case (as a matter of fact, we shall use Lemma \ref{Lemma:SchwarzLemma} only for the case of functions of a single variable). 

\begin{lemma}\label{Lemma:SchwarzLemma}
If $f$ has a zero of multiplicity $\ge h$ at the origin, then for $R'\le R$ we have
$$
|f|_{R'}\le \left(\frac{R'}{R}\right)^h |f|_R.
$$
\end{lemma}

\begin{proof} [Proof (following \cite{MR1942136}).]
Let $\uz$ satisfy $|f(\uz)|= |f|_R$ and $ |z_i|\le R$. Define $g(t)=t^{-h}f(t\uz)$ for $t\in \bQ_p$ with $|t|\le 1$.  Since $R'/R\le 1$, we deduce  $|g|_{ R'/R}\le |g|_1$.  Since 
$ |g|_1=|f|_R$ and $|g|_{ R'/R}=(R/R')^h|f|_R$, 
Lemma \ref{Lemma:SchwarzLemma} follows.
\end{proof}

A quantitative version of Lemma \ref{Lemma:SchwarzLemma} is Lemma 3.4.p of  \cite{MR82k:10042}.

\begin{corollary}\label{Corollary:UpperBoundInterpolationDeterminant}
Let $f_1,\ldots,f_L$ be power series in $\bQ_p^r $ with $|f_\lambda|_R<\infty$ and let $\uz_1,\ldots,\uz_L$ be points in the polydisc  $ |z_i|\le R'$ with $R'\le R$. Then the determinant
$$
\Delta=\det
\Bigl(f_\lambda(\uz_\mu)\Bigr)_{1\le\lambda,\mu\le L}
$$
is bounded by 
$$
|\Delta|\le 
L! 
\left(\frac{R'}{R}\right)^{L^{1+1/r} } \prod_{\lambda=1}^L  |f_\lambda|_R.
$$
\end{corollary}

\begin{proof}
Corollary \ref{Corollary:UpperBoundInterpolationDeterminant} is an ultrametric version of Lemma 6.3 of  \cite{MR1756786}; it follows from Lemma \ref{Lemma:SchwarzLemma} by means of Lemmas   6.4 and 6.5 of \cite{MR1756786}, according to which the function of one variable 
$$
\Psi(t)=\det
\Bigl(f_\lambda(t\uz_\mu)\Bigr)_{1\le\lambda,\mu\le L}
$$
has a zero of multiplicity  greater than  $(n/e) L^{1+1/n}$ at the origin.

\end{proof}

\subsubsection{$p$--adic theta functions}\label{Subsubsection:p-adicTheta}

Since the kernel of the logarithmic map 
$$
\log: A(\bQ_p)\longrightarrow T_A(\bQ_p)
$$ 
is the set of torsion points of $ A(\bQ_p)$, this map is locally injective near the neutral element of $A(\bQ_p)$. Let $\calU$ be an open neighborhood of $(1:0:\cdots:0)$ in $A(\bQ_p)$, $\calV$ be an open neighborhood of $0$ in $T_A(\bQ_p)$  and $\theta:\calV\rightarrow \calU$ be a local inverse of $\log$:
$$
u\in\calU\; \Longrightarrow \; \log u\in\calV \quad  \hbox{and} \quad 
 \theta \log(u)=u,
$$
$$
v\in\calV\; \Longrightarrow \; \theta(v)\in\calU \quad  \hbox{and} \quad 
 \log  \theta(v)=v.
 $$
By definition of $r$, $\overline{\log \Gamma}$ is a $\bZ_p$-submodule of $ T_A(\bQ_p)$ of dimension $r$  which contains the $\ell$ elements $\log \gamma_j $ ($1\le j\le \ell$). Let $e_1,\ldots,e_r$ be a basis. 
Let $R>0$ be a positive real number such that  $z_1e_1+\cdots+z_re_r\in \calV$  for any $\uz=(z_1,\ldots,z_r)\in \bQ_p^r$ with $|z_i|_p \le R$.  
For $\uz=(z_1,\ldots,z_r)\in \bQ_p^r$ with $|z_i|_p <R$, define $\theta_1(\uz),\ldots,\theta_N(\uz)$ by 
$$
\theta(z_1e_1+\cdots+z_re_r)=\bigl(1:\theta_1(\uz):\cdots:\theta_N(\uz)\bigr).
$$
Then $\theta_1,\ldots,\theta_N$ are power series in $r$ variables with coefficients in $\bQ_p$ and  radius of convergence $\ge R$.

Write 
$$
\log \gamma_j =\sum_{i=1}^r \eta_{ji} e_i \quad\hbox{and}\quad 
y_j=(\eta_{j1},\ldots,\eta_{jr})\in \bQ_p^r \quad (1\le j\le \ell).
$$
Further,  select $M\in\bZ_{>0}$ such that 
$$
\max_{\atop{1\le j\le \ell}{1\le i\le r}} | M \eta_{ji} |_p<R.
$$
Then, for any $\us\in\bZ^{\ell}$ with $M|s_j$ for $1\le j\le \ell$, 
$$
s_1\log\gamma_1+\cdots+s_\ell\gamma_\ell\in\calV
$$
and
$$
\theta(s_1\log\gamma_1+\cdots+s_\ell \log \gamma_\ell)=\bigl(1:\gamma_{\us 1}:\cdots:\gamma_{\us N}\bigr).
$$  
 Hence 
 $\gamma_{\us \nu}=\theta_\nu(y_{\us})$ for all  $\us\in\bZ^\ell$  with $M|s_j$ and  for all $\nu$ with $1\le \nu\le N$. 

\section{The zero estimate and the interpolation determinant}\label{Section:ZeroEstimate}

The zero--estimate of Masser--Wüstholz (Main Theorem of  \cite{MR632987})  is valid for a quasi--projective commutative algebraic group variety  over a field $K$ of zero characteristic. We need it only for a simple Abelian variety, which makes the statement  shorter, since there is no algebraic subgroup to worry about.

Let again $A$ be a simple Abelian variety of dimension $g$ embedded into a projective space $\bP_N$. When $P\in K[Y_0,\ldots,Y_N]$ is a non--zero  homogenous polynomial, we denote by $Z(P)$  the hypersurface $P=0$ of $\bP_N(K)$.

\begin{lemma}[Zero estimate]\label{Lemma:ZeroEstimate}
There exists a constant $c>0$ depending only on $A$ and on the embedding of $A$ into $\bP_N$ with the following property. 
Let $\gamma_1, \ldots,\gamma_\ell$ be $\bZ$--linearly independent elements in  $ A(K)$.  Let 
 $P\in K[Y_0,\ldots,Y_N]$ be a homogenous polynomial of total degree $\le D$, such  that $Z(P)$ does not contain $A(K)$ but contains 
 $$
 \Gamma(S)=\bigl\{s_1\gamma_1+\ldots+s_\ell\gamma_s\; ; \; \us\in\bZ^{\ell}(S)\bigr\}.
 $$ 
 Then 
$$
D>c
(S/g)^{\ell/g} .
$$

\end{lemma}

Like in \cite{MR84k:10032}, \S~2.b, we could replace the zero estimate by an interpolation lemma due to D.W.~Masser
(\cite{MR702196} and Theorem 2.1 of \cite{MR84k:10032}). The idea is just to consider the transposed matrix. 

Coming  back to the notations of \S~\ref{Section:FurtherNotations} (recall in particular the integer $M>0$ introduced in \S~\ref{Subsubsection:p-adicTheta}), we deduce from Lemma \ref{Lemma:ZeroEstimate}:

\begin{corollary}\label{Corollary:InterpolationDeterminant}
There exist two integers $c_1>1$ and $N_0>1$, depending on  $A$ and $\gamma_1,\ldots,\gamma_\ell$, with the following property: if $N$ is  a positive integer with $N\ge N_0$ and if we set 
$$
L=N^{\ell g},\quad T=N^\ell, \quad S=c_1 N^g,
$$
then there exists  a subset  $\calS=\{ \us_1,\ldots,\us_L\}$ of $\bZ^{\ell}(S)$ with $L$ elements $\us_\mu=(s_{\mu,j})_{1\le j\le \ell}$ ($1\le \mu\le L$), such that $M|s_{\mu,j}$ for $1\le j\le \ell$ and  $1\le \mu\le L$, 
and such that  the determinant
$$
\Delta=\det 
\left(
\gamma_{\us 1}^{t_1}\cdots\gamma_{\us g}^{t_g}
\right)_{
\us\in \calS , \; \ut\in\bZ^g(T)}
$$
does not vanish.

\end{corollary}

\begin{proof}
Consider the 
matrix
$$
\left(
\gamma_{\us 1}^{t_1}\cdots\gamma_{\us g}^{t_g}
\right)_{\ut,\us},
$$
where the index of rows is  $ \ut\in\bZ^g(T)$, while the index of columns $\us$ runs over the elements in $ \bZ^\ell(S)$ for which $M$ divides $s_j$.
Our goal is to prove that this matrix 
has maximal rank $L$.
Consider a system of relations  among the rows of the matrix 
$$
\sum_{\ut\in\bZ^g(T)} p_{\ut} \gamma_{\us 1}^{t_1}\cdots\gamma_{\us g}^{t_g}=0 \qquad (\us\in \bZ^\ell(S),\; M|s_j)
$$
with $p_{\ut } \in k$ for all $\ut\in\bZ^g(T)$.
The polynomial 
$$
\sum_{\ut\in\bZ^g(T)} p_{\ut}  X_1^{t_1}\cdots X_g^{t_g}
$$
has degree $\le T$ in each of the variables $X_1,\ldots,X_g$ and vanishes at all points of $\gamma_{\us}\in \Gamma(S)$ for which $M|s_j$ ($1\le j\le \ell$). Use Lemma  \ref{Lemma:ZeroEstimate} with $\gamma_1,\ldots,\gamma_\ell$ replaced by
$M\gamma_1,\ldots,M\gamma_\ell$.  Taking $c_1>M g (g/c)^{g/\ell}$, so that $gN^\ell < c (c_1N^g/ g M)^{\ell/g}$,  it follows that 
 this polynomial is $0$, hence $p_{\ut}=0$ for all $\ut \in\bZ^g(T)$. 
\end{proof}

\section{Upper bound for the height and lower bound for the absolute value of the interpolation determinant}\label{Section:LowerBound}

Under the assumptions of Theorem \ref{Theorem:MainTranscendenceResult}, we give an upper bound for the height of the determinant $\Delta$ introduced in Corollary \ref{Corollary:InterpolationDeterminant}.

\begin{proposition}\label{Proposition:UpperBoundHeightInterpolationDeterminant}
There exists a positive integer $c_2>1$, depending on   $A$ and $\gamma_1,\ldots,\gamma_\ell$,  such that, for all $N\ge N_0$, 
$$
\rmh(\Delta)\le c_2LTS^2.
$$

\end{proposition}

\begin{proof}
From Lemma \ref{Lemma:NeronTateHeight}, we deduce, for any $\us\in\bZ^\ell(S)$, 
$$ 
\rmh
\left(1: 
 \gamma_{\us 1}:\cdots:
 \gamma_{\us N}
\right)
\le c S^2.
$$
Proposition \ref{Proposition:UpperBoundHeightInterpolationDeterminant} now follows from Lemma \ref{Lemma:Height} with 
$$
\nu_1=\cdots=\nu_L=g, \quad
N_1=\cdots=N_L=T 
\quad \hbox{ and  } \quad \rmL(f)\le L!
$$
\end{proof}

Liouville's inequality  (Lemma~\ref{Lemma:InegaliteLiouville}) implies:

\begin{corollary}\label{Corollary:LowerBoundInterpolationDeterminant}
With the notations of Proposition \ref{Proposition:UpperBoundHeightInterpolationDeterminant},
$$
\log |\Delta|_p\ge -c_2LTS^2.
$$ 

\end{corollary}

\section{Analytic estimate: upper bound for the  absolute value of the interpolation determinant}\label{Section:UpperBound}

\begin{proposition}\label{Proposition:UpperBoundInterpolationDeterminant}
There exists a positive integer $c_3>1$, depending on   $A$ and $\gamma_1,\ldots,\gamma_\ell$,  such that, for all $N\ge N_0$, 
$$
\log |\Delta|_p \le -c_3L^{1+1/r}.
$$
\end{proposition}

\begin{proof}
Proposition \ref{Proposition:UpperBoundInterpolationDeterminant} 
follows from Corollary   \ref{Corollary:UpperBoundInterpolationDeterminant} with the set of functions 
$$
\{f_1,\ldots,f_L\}=
\left\{
\theta_1^{t_1}\cdots \theta_g^{t_g}\; ;\; 
\ut\in\bZ_{\ge 0}(T)
\right\}
$$
and the points $\uz_\mu=s_{\mu 1} y_1+\cdots+s_{\mu \ell} y_\ell$
($1\le \mu\le L$).
\end{proof}

\section{Proof of  the main transcendence result}\label{Section:Proof}

\begin{proof}[Proof ot Theorem \ref{Theorem:MainTranscendenceResult}]

Since $TS^2=c_1^2 L^{(1/g)+(2/\ell)}$, the conclusions of corollary \ref{Corollary:LowerBoundInterpolationDeterminant} and proposition \ref{Proposition:UpperBoundInterpolationDeterminant} imply 
$$ 
\frac{1}{r}\le \frac{1}{g}+\frac{2}{\ell} \cdotp
$$
\end{proof}

\section{Remarks}\label{Section:Remarks}

\null\hskip - 1 true cm 
$\bullet$
{\bf \ref{Section:Remarks}.1.} 
In place of the rational number field and the prime number $p$, one may work with an algebraic number field and a finite place $v$, replacing $\bQ_p$ with the completion $k_v$. 
One main difference is in \S~\ref{Subsubsection:Height},
where, in the case of a number field, one needs to introduce  height functions on the field of algebraic numbers in place of the rational number field. 
See \cite{MR1756786} Chap.~3 \S~2, \cite{MR1745599}, \S~B.2, \cite{MR1757192},  Chap.~2,  \cite{MR2216774} Chap.~1, \cite{MR2465098} Chap.~4.

As pointed out in \cite{Poonen} (Remark 6.4), one cannot  deduce the general case of a number field from the special case of the rational numbers by means of the restriction of scalars. 

\noindent
\null\hskip - 1 true cm 
$\bullet$
{\bf \ref{Section:Remarks}.2.} 
As mentioned in  \cite{MR84j:10046} (\S~6a p.~643), similar results hold when the simple Abelian variety $A$ is replaced by a commutative algebraic group $G$. There is a condition in  \cite{MR84j:10046} for the ultrametric case that a subgroup of finite index of $\Gamma$ is contained in a compact subgroup of $A(k_v)$ -- for an Abelian variety $A$, the group  $A(k_v)$ is compact and this condition is always satisfied.

Let us write, like in \cite{MR84j:10046}, $G=\bG_A^{d_0}\times \bG_m^{d_1} \times  G'$, where $G'$ has dimension $d_2$ (and therefore  $G$ has dimension $d=d_0+d_1+d_2$). Roughly speaking,  in this general sitting, one replaces 
$$
\frac{\ell g}{\ell+2g}
\quad\hbox{by}\quad
\frac{\ell d}{\ell+d_1+2d_2}\cdotp
$$
However, one needs to take into account possible degeneracies occurring from the algebraic subgroups of $G$. We refer to \cite{MR84j:10046} for precise statements. 

In the case of a power of the multiplicative group $G= \bG_m^{d}$, the transcendence result yields lower bounds for the $p$--adic rank of the units of an algebraic number field (namely partial results towards Leopoldt's Conjecture).

\noindent
\null\hskip - 1 true cm 
$\bullet$
{\bf \ref{Section:Remarks}.3.}  
Following \cite{Poonen}, consider a commutative  algebraic group $G$ over $\bQ$ and a finitely generated subgroup $\Gamma$ of $G(\bQ)$ contained in the union of compact subgroups of $\bG(\bQ_p)$. The number $\dim(\Gammabar)$ can be defined exactly like in \S~\ref{Subsection:p-adicLogarithm} as the dimension of the $\bZ_p$--submodule of the tangent space at the origine $\Lie(G)$ spanned by the image of $\Gamma$ under the logarithmic map. Another function $d(\Gamma)$ of $\Gamma$ is introduced by B.~Poonen in    \cite{Poonen}:
$$
d(\Gamma):=\min_{H\subset G}\bigl\{  \dim H+\rk_\bZ(\Gamma/\Gamma\cap H ) \bigr\},
$$
where the minimum is over all algebraic subgroups $H$ of $G$ over $\bQ$.
The inequality 
$\dim(\Gammabar)\le d(\Gamma)$ is always true. 
Here is an example where this inequality is strict (compare with  Langevin's example in \cite{MR1428494} p.~1201 and 1209 for $\bG_m^3$). 
Consider an elliptic curve $E$ over $\bQ$ with  three linearly independent algebraic points  $\gamma_1, \gamma_2, \gamma_3$ in $E(\bQ)$. Let  $\Gamma$  be the subgroup of $E^3(\bQ)$ generated by $(0,\gamma_3, -\gamma_2)$, $(-\gamma_3,0, \gamma_1)$, $(\gamma_2, -\gamma_1,0)$. Then  $\dim\Gammabar=2$, 
while  $d(\Gamma)=3$. 

To produce a lower bound for the $p$--adic rank amounts to produce lower bounds for the rank of certain matrices whose entries are $p$--adic logarithms of algebraic points. From a conjectural point of view, the answer is given by the {\it structural rank} introduced  by D.~Roy. See \cite{MR1756786} for the case of linear algebraic groups.

\noindent
\null\hskip - 1 true cm 
$\bullet$
{\bf \ref{Section:Remarks}.4.} 
Further applications of the algebraic subgroup theorem in the ultrametric case are given by D.~Roy in \cite{MR1368440}.

\noindent
\null\hskip - 1 true cm 
$\bullet$
{\bf \ref{Section:Remarks}.5.} 
Our $p$--adic result Theorem \ref{Theorem:MainTranscendenceResult} is an ultrametric version of \cite{MR84j:10046,MR1341724} (see also \cite{MR1245821}). In the Archimedean case, quantitative refinements are given in \cite{MR1711387}, they are based on the results of  \cite{MR1491809}. See also \cite{arXiv:1007.0593}. Since the method is ``effective'', it is also possible to produce quantitative  refinements of  Theorem \ref{Theorem:MainTranscendenceResult}.

\noindent
\null\hskip - 1 true cm 
$\bullet$
{\bf \ref{Section:Remarks}.6.} 
An alternative proof of the main result (Theorem \ref{Theorem:MainTranscendenceResult}) can be given by means of  Arakelov's geometry and Bost slope inequality. See the papers by J.B.~Bost \cite{MR1423622} and A.~Chambert--Loir \cite{MR1975179}.

\bibliographystyle{siam} %


\begin{thebibliography}{10}

\bibitem{MR990016}
{\sc Y.~Andr{\'e}}, {\em {$G$}-functions and geometry}, Aspects of Mathematics,
  E13, Friedr. Vieweg \& Sohn, Braunschweig, 1989.

\bibitem{MR2115000}
\leavevmode\vrule height 2pt depth -1.6pt width 23pt, {\em Une introduction aux
  motifs (motifs purs, motifs mixtes, p\'eriodes)}, vol.~17 of Panoramas et
  Synth\`eses, Soci\'et\'e Math\'ematique de France, Paris, 2004.

\bibitem{Bellaiche}
{\sc J.~Bellaiche}, {\em Personal communication}.
\newblock February 2010.

\bibitem{MR2216774}
{\sc E.~Bombieri and W.~Gubler}, {\em Heights in {D}iophantine geometry},
  vol.~4 of New Mathematical Monographs, Cambridge University Press, Cambridge,
  2006.

\bibitem{MR1423622}
{\sc J.-B. Bost}, {\em P\'eriodes et isogenies des vari\'et\'es ab\'eliennes
  sur les corps de nombres (d'apr\`es {D}. {M}asser et {G}. {W}\"ustholz)},
  Ast\'erisque,  (1996), pp.~115--161.
\newblock S{\'e}minaire Bourbaki, Exp.\ No.\ 795, 4, Vol. 1994/95.

\bibitem{MR1728312}
{\sc N.~Bourbaki}, {\em Lie groups and {L}ie algebras. {C}hapters 1--3},
  Elements of Mathematics (Berlin), Springer-Verlag, Berlin, 1998.
\newblock Translated from the French, Reprint of the 1989 English translation.

\bibitem{MR1975179}
{\sc A.~Chambert-Loir}, {\em Th\'eor\`emes d'alg\'ebricit\'e en g\'eom\'etrie
  diophantienne (d'apr\`es {J}.-{B}.\ {B}ost, {Y}.\ {A}ndr\'e, {D}. \& {G}.\
  {C}hudnovsky)}, Ast\'erisque,  (2002), pp.~175--209.
\newblock S{\'e}minaire Bourbaki, Exp. No. 886, viii, Vol. 2000/2001.

\bibitem{arXiv:1007.0593}
{\sc A.~Ghosh, A.~Gorodnik, and A.~Nevo}, {\em Diophantine approximation and
  automorphic spectrum}.
\newblock July 4, 2010\\ \href{http://arxiv.org/abs/1007.0593}{\tt
  http://arxiv.org/abs/1007.0593}.

\bibitem{MR1745599}
{\sc M.~Hindry and J.~H. Silverman}, {\em Diophantine geometry}, vol.~201 of
  Graduate Texts in Mathematics, Springer-Verlag, New York, 2000.
\newblock An introduction.

\bibitem{MR2465098}
{\sc P.-C. Hu and C.-C. Yang}, {\em Distribution theory of algebraic numbers},
  vol.~45 of de Gruyter Expositions in Mathematics, Walter de Gruyter GmbH \&
  Co. KG, Berlin, 2008.

\bibitem{MR702196}
{\sc D.~W. Masser}, {\em Interpolation on group varieties}, in Diophantine
  approximations and transcendental numbers (Luminy, 1982), vol.~31 of Progr.
  Math., Birkh\"auser Boston, Mass., 1983, pp.~151--171.

\bibitem{MR632987}
{\sc D.~W. Masser and G.~W{\"u}stholz}, {\em Zero estimates on group varieties.
  {I}}, Invent. Math., 64 (1981), pp.~489--516.

\bibitem{MR1330932}
{\sc B.~Mazur}, {\em Speculations about the topology of rational points: an
  update}, Ast\'erisque,  (1995), pp.~165--182.
\newblock Columbia University Number Theory Seminar (New York, 1992).

\bibitem{Poonen}
{\sc B.~Poonen}, {\em The $p$--adic closure of a subgroup of rational points on
  a commutative algebraic group, {\rm 2006 (unpublished)}\hfill}.
\newblock \href{http://www-math.mit.edu/~poonen/papers/leopoldt.pdf}{\tt
  http://www-math.mit.edu/~poonen/papers/leopoldt.pdf}.

\bibitem{MR1368440}
{\sc D.~Roy}, {\em On the {$v$}-adic independence of algebraic numbers}, in
  Advances in number theory (Kingston, ON, 1991), Oxford Sci. Publ., Oxford
  Univ. Press, New York, 1993, pp.~441--451.

\bibitem{MR1245821}
{\sc D.~Roy and M.~Waldschmidt}, {\em Autour du th\'eor\`eme du sous-groupe
  alg\'ebrique}, Canad. Math. Bull., 36 (1993), pp.~358--367.

\bibitem{MR1942136}
{\sc J.-P. Serre}, {\em Dépendance d'exponentielles $p$--adiques}, Séminaire
  Delange--Pisot--Poitou. Théorie des Nombres, t.~{\bf 7}, n~$^\circ$2
  (1965--66), exp.~n$^\circ$~15,  (1966), pp.~1--14\hfill.
\newblock
  \href{http://archive.numdam.org/article/SDPP_1965-1966__7_2_A4_0.pdf}{\tt
  \small
  http://archive.numdam.org/article/SDPP$_-$1965-1966${_-}{_-}$7$_-$2$_-$A4$_-%
$0.pdf}.

\bibitem{MR1757192}
\leavevmode\vrule height 2pt depth -1.6pt width 23pt, {\em Lectures on the
  {M}ordell-{W}eil theorem}, Aspects of Mathematics, E15, Friedr. Vieweg \&
  Sohn, Braunschweig, (first edition 1989) third~ed., 1997.
\newblock Translated from the French and edited by Martin Brown from notes by
  Michel Waldschmidt, With a foreword by Brown and Serre.

\bibitem{MR82k:10042}
{\sc M.~Waldschmidt}, {\em Transcendance et exponentielles en plusieurs
  variables}, Invent. Math., 63 (1981), pp.~97--127.

\bibitem{MR84k:10032}
\leavevmode\vrule height 2pt depth -1.6pt width 23pt, {\em D\'ependance de
  logarithmes dans les groupes alg\'ebriques}, in Diophantine approximations
  and transcendental numbers (Luminy, 1982), vol.~31 of Progr. Math.,
  Birkh\"auser Boston, Mass., 1983, pp.~289--328.

\bibitem{MR84j:10046}
\leavevmode\vrule height 2pt depth -1.6pt width 23pt, {\em Sous-groupes
  analytiques de groupes alg\'ebriques}, Ann. of Math. (2), 117 (1983),
  pp.~627--657.

\bibitem{MR1341724}
\leavevmode\vrule height 2pt depth -1.6pt width 23pt, {\em Densit\'e des points
  rationnels sur un groupe alg\'ebrique}, Experiment. Math., 3 (1994),
  pp.~329--352.
\newblock Erratum: vol.~4 (3) 1995), 255.

\bibitem{MR1428494}
\leavevmode\vrule height 2pt depth -1.6pt width 23pt, {\em Dependence of
  logarithms on commutative algebraic groups}, Rocky Mountain J. Math., 26
  (1996), pp.~1199--1223.
\newblock Symposium on Diophantine Problems (Boulder, CO, 1994).

\bibitem{MR1491809}
\leavevmode\vrule height 2pt depth -1.6pt width 23pt, {\em Approximation
  diophantienne dans les groupes alg\'ebriques commutatifs. {I}. {U}ne version
  effective du th\'eor\`eme du sous-groupe alg\'ebrique}, J. Reine Angew.
  Math., 493 (1997), pp.~61--113.

\bibitem{MR1711387}
\leavevmode\vrule height 2pt depth -1.6pt width 23pt, {\em Density measure of
  rational points on abelian varieties}, Nagoya Math. J., 155 (1999),
  pp.~27--53.

\bibitem{MR1756786}
\leavevmode\vrule height 2pt depth -1.6pt width 23pt, {\em Diophantine
  approximation on linear algebraic groups}, vol.~326 of Grundlehren der
  Mathematischen Wissenschaften [Fundamental Principles of Mathematical
  Sciences], Springer-Verlag, Berlin, 2000.
\newblock Transcendence properties of the exponential function in several
  variables.

\end{thebibliography}

	  

\vfill

 \vskip 2truecm plus .5truecm minus .5truecm 

\font\ninerm=cmr10 at 9pt

\font\ninett=cmtt10 scaled 900 

\hfill
\vbox{\ninerm
 \hbox{Michel WALDSCHMIDT}
 \hbox{Universit\'e P.~et M.~Curie (Paris VI)}
 \hbox{Institut de Math\'ematiques de Jussieu -- CNRS UMR 7586}           
 \hbox{4, Place Jussieu } 
 \hbox{F--75252 PARIS Cedex 05 FRANCE}
 \hbox{e-mail: 
\href{mailto:miw@math.jussieu.fr}{miw@math.jussieu.fr}}
 \hbox{
URL:
\href{http://www.math.jussieu.fr/~miw/}%
{http{$:$}//www.math.jussieu.fr/${\scriptstyle \sim}$miw/}}
} 

\vfill

\end{document}